\documentclass[final]{siamart251216}
\usepackage{amsmath,amssymb,amsfonts,graphicx,bm} 
\usepackage{subcaption}
\usepackage{hyperref}

\newsiamremark{remark}{Remark}
\newsiamthm{assumption}{Assumption}

\crefname{remark}{Remark}{Remarks}
\crefname{assumption}{Assumption}{Assumptions}

\usepackage{algorithm}
\usepackage[noend]{algorithmic}

\DeclareMathOperator*{\argmin}{arg\,min}
\title{Convergence of high-index saddle dynamics for degenerate saddle points on critical manifolds\thanks{Submitted to.
\funding{L.Z. was supported by the National Natural Science Foundation of China (No. 12225102, T2321001, 12288101) and National Key Research and Development Program of China (No. 2024YFA0919500). 
L.T. was supported by the National Key Research and Development Program of China (No. 2022YFA1008200) and Shanghai Institute for Mathematics and Interdisciplinary Sciences (SIMIS) under grant SIMIS-ID-2025-ST.
J.Y. was supported by ``the Fundamental Research Funds for the Central Universities'' (No.~2253100016).}}}

\author{
Tao Luo\thanks{School of Mathematical Sciences, Institute of Natural Sciences, MOE-LSC, CMA-Shanghai, Shanghai Jiao Tong University, Shanghai 200240,
China (\email{luotao41@sjtu.edu.cn}).}
\and
Jianyuan Yin\thanks{School of Mathematical Sciences, 
Laboratory of Mathematics and Complex Systems, Ministry of Education, Beijing Normal University, Beijing 100875, China (\email{jyyin@bnu.edu.cn}).}
\and
Lei Zhang\thanks{School of Mathematical Sciences, Beijing International Center for Mathematical Research, Center for Quantitative Biology, Center for Machine Learning Research, Peking University, Beijing 100871, China (\email{zhangl@math.pku.edu.cn}).}
\and
Shixue Zhang\thanks{School of Mathematical Sciences, Peking University, Beijing 100871, China (\email{zhangshixue@stu.pku.edu.cn}).} 
} 

\begin{document}
\maketitle
\begin{abstract}
The high-index saddle dynamics (HiSD) method provides a powerful framework for finding saddle points and constructing solution landscapes. 
While originally derived for nondegenerate critical points, HiSD has demonstrated empirical success in degenerate cases, where the Hessian matrix exhibits zero eigenvalues. 
However, the mathematical and numerical analysis of HiSD for degenerate saddle points remains unexplored. 
In this paper, utilizing Morse--Bott functions, we present a rigorous analysis of HiSD for computing degenerate saddle points on a critical manifold.
We prove the local convergence of the continuous HiSD and establish the linear convergence rate of the discrete HiSD algorithm. 
Furthermore, we provide a theoretical explanation for the gradient alignment tendency, revealing that the gradient direction asymptotically aligns with a specific Hessian eigenvector. 
Our analysis also elucidates the flexibility in selecting the index for HiSD in the context of degenerate saddle points.
We validate our analytical results through numerical experiments on neural-network loss landscapes and demonstrate that momentum-accelerated variants of HiSD achieve rapid convergence to degenerate saddle points.
\end{abstract}



\pagestyle{myheadings}
\thispagestyle{plain}

\markboth{Luo, Yin, Zhang and Zhang}{Convergence of high-index saddle dynamics for degenerate saddle points on critical manifolds}

\section{Introduction}
The computation of saddle points has long been an important topic across various scientific fields. Saddle points on smooth energy landscapes play a crucial role in phase transitions \cite{prl2010, wang2010phase, samanta2014microscopic, zhang2016recent}, chemical reactions \cite{heyden2005efficient, jacobson2017automated}, and training neural networks in machine learning \cite{NIPS2014, daneshmand2018, zhang2021embedding}. 
However, locating saddle points is a challenging problem due to  their unstable nature, especially in high-dimensional systems where the transition mechanism is unknown.

Various numerical methods have been developed to compute saddle points, which can broadly be classified into two categories.
Path‐finding methods, such as string methods \cite{weinan2002string} and nudged elastic band methods \cite{henkelman2000improved}, aim to find the minimum energy path and transition state between initial and final states.
In contrast, surface‐walking methods aim to locate the saddle point starting from an initial guess.
Examples of surface‐walking methods for index-1 saddle points include gentlest ascent dynamics \cite{gad}, dimer-type methods \cite{Henkelman1999,zhang2012shrinking}, and the iterative minimization formulation \cite{gao2015iterative}. 
Recently, the high‐index saddle dynamics (HiSD) method \cite{yin2019high} has provided a unified surface‐walking framework for searching for index-$k$ saddle points, paving the way for various extensions and variants \cite{yin2021searching, luo2025accelerated, su2025improved}. 
The solution landscape, which consists of all stationary points and their connections, can be constructed by combining the HiSD method with an upward/downward search algorithm \cite{YinPRL, yin2021searching, zhang2023review}. 
The solution landscape method has shown its effectiveness in many scientific applications, including liquid crystals \cite{YinPRL, han2021solution}, nucleation of quasicrystals \cite{yin2021transition}, reaction-diffusion systems \cite{WU2025reaction}, and Bose--Einstein condensates \cite{yin2024revealing}.

For the HiSD method, it is standard practice to assume that the target saddle point is nondegenerate during derivation and theoretical analysis.
A saddle point is termed nondegenerate if the Hessian (second-order derivative) of the energy function at this point has a bounded inverse, \textit{i.e.} all the eigenvalues of the Hessian are nonzero. 
According to classical Morse theory \cite{milnormorse}, nondegenerate saddle points are inherently isolated, and Morse's lemma guarantees that around such points, the energy function behaves locally like a nondegenerate quadratic form.
Under the nondegeneracy assumption, the saddle point can be characterized as the solution to a minimax problem. Mathematical and numerical analyses of HiSD have established local stability and convergence \cite{yin2019high, luo2022convergence, zhang2022sinum, su2025improved}.

However, for energy functions in many practical applications, the stationary points, including both saddle points and minima, are not isolated but form a continuous manifold, often referred to as a critical manifold, which leads to degeneracy.
For example, in physical systems, when a continuous symmetry is broken to a discrete symmetry, Goldstone modes emerge \cite{goldstone1962broken}, the number of which equals the multiplicity of zero eigenvalues. 
Notably, in over-parameterized neural networks, empirical studies reveal that during training, a large fraction of the Hessian eigenvalues cluster near zero \cite{sagun2018}; consequently, the loss landscape exhibits a significant degree of degeneracy. 
Furthermore, as shown by the parameter condensation phenomenon during training \cite{luo2021phase}, the intrinsic symmetries of neural networks lead to the existence of degenerate saddle points \cite{zhang2021embedding}.

To mathematically describe such structures, the theory of Morse--Bott functions extends classical Morse theory to these degenerate settings, allowing critical points to form smooth manifolds rather than isolated points \cite{banyaga2010morse}. 
Consistent with classical Morse theory \cite{milnormorse}, the index of a critical point is defined as the negative inertia index of the Hessian, while the nullity corresponds to the multiplicity of zero eigenvalues.
This framework provides a natural language for analyzing energy landscapes in complex systems and has been applied to the loss landscapes of neural networks \cite{zhang2025geometry}.

Although the derivation and analyses of the HiSD method are based on the assumption of nondegeneracy, this method has successfully identified degenerate saddle points in various scientific problems \cite{yin2021transition, yin2022solution, yin2022constrained, yin2024revealing}. 
By treating those zero eigenvectors as unstable directions, the HiSD method can successfully climb out of a critical manifold of minima and locate a degenerate saddle point, while many surface-walking methods for index-1 saddle points fail to do so due to their inability to accommodate zero eigenvectors.
Specifically, to search for an index-$s$ saddle point with nullity $m$, numerical evidence suggests that one can employ HiSD with different indices ranging from $s$ to $(s+m)$. 
However, a rigorous theoretical explanation for why HiSD with these indices converges to degenerate saddle points, including stability analysis and convergence rates, remains absent from previous literature.

In this paper, we aim to establish a theoretical framework for applying HiSD to degenerate saddle points by considering Morse--Bott functions as a generalization of Morse functions. 
Specifically, for a saddle point located on a critical manifold that is nondegenerate in the normal space, we derive the following results under some assumptions:
\begin{itemize}
\item We establish the local asymptotic stability of the critical manifold under continuous HiSD and prove its convergence to a point on the manifold.
\item We prove that the discrete HiSD algorithm with accurate eigenvectors converges to the critical manifold at a linear convergence rate. 
\item We provide a rigorous justification for the gradient alignment tendency, where the gradient direction aligns with a specific Hessian eigenvector asymptotically during iterations.
\item We demonstrate the flexibility of the HiSD index, showing that the theoretical results hold for indices ranging from $s$ to $s+m$, where $s$ is the index and $m$ is the nullity.
\item We validate these findings through numerical experiments and show that momentum-accelerated variants of HiSD achieve rapid convergence to degenerate saddle points.
\end{itemize}

The rest of the paper is organized as follows. 
In \Cref{sec:2}, we provide the definition of Morse--Bott functions and review the continuous dynamics and the discrete algorithms of HiSD.
In \Cref{sec:3}, we prove asymptotic stability and convergence results and explain the gradient alignment tendency.
In \Cref{sec:4}, we present numerical experiments to validate our theoretical findings and demonstrate the high efficiency of momentum-accelerated variants of HiSD in degenerate settings. 
Finally, we present concluding remarks in \Cref{sec:5}, with the proofs of lemmas provided in Appendix \ref{app}.

\section{Preliminaries} \label{sec:2}
\subsection{Morse--Bott function}
For a smooth function $E:\mathbb{R}^d\to\mathbb{R}$, a point $\theta \in \mathbb{R}^d$ is called a \textit{critical point} if the gradient vanishes, i.e., $\nabla E(\theta) = 0$. 
The \textit{(Morse) index} of a critical point, defined as the number of negative eigenvalues of its Hessian $\nabla^2 E$, characterizes the local nature of the critical point. 
The \textit{nullity} of a critical point is defined as the multiplicity of zero eigenvalues of its Hessian $\nabla^2 E$, \textit{i.e.}, the dimension of the kernel of the Hessian. 
A critical point is said to be \textit{nondegenerate} if its Hessian at that point has a bounded inverse, i.e., the Hessian has no zero eigenvalues. 
This leads to the following definition of Morse functions \cite{milnormorse}, which constitute the foundation of Morse theory and play a central role in the analysis of gradient flows.
\begin{definition}[Morse function]
A smooth function $E$ is called a \emph{Morse function} if all its critical points are nondegenerate.
\end{definition}

To handle degenerate cases, Morse–Bott theory generalizes classical Morse theory by accommodating critical submanifolds in place of isolated critical points \cite{banyaga2010morse, bott1954}.
Denote by $\mathcal{M}_E$ the set of all critical points of a smooth function $E:\mathbb{R}^d\to\mathbb{R}$, and let $\mathcal{M}$ be a connected component of $\mathcal{M}_E$.
The key insight of Morse--Bott theory is to examine how the Hessian behaves on the critical submanifold $\mathcal{M}$. 
Based on the Euclidean metric in $\mathbb{R}^d$, the entire space at each point $\theta\in \mathcal{M}$ admits the orthogonal decomposition $\mathbb{R}^d=T_{\theta} \mathcal{M}\oplus N_{\theta} \mathcal{M}$, where $N_{\theta} \mathcal{M}$ is the normal space to $\mathcal{M}$ at $\theta$.

An important fact is that the tangent space $T_{\theta} \mathcal{M}$ at the critical point $\theta$ lies entirely in the kernel of the Hessian, which can be derived straightforward. 
For any ${w}\in T_{\theta} \mathcal{M}$, there exists a smooth curve $\gamma:(-\varepsilon, \varepsilon)\to \mathcal{M}$ satisfying $\gamma(0)=\theta$ and $\dot{\gamma}(0)={w}$. 
Since $\nabla E(\gamma(t))={0}$ for all $t \in (-\varepsilon, \varepsilon)$, we have $\nabla^2 E(\gamma(t))\dot{\gamma}(t) = {0}$ by taking derivatives with respect to $t$.
Therefore, at $t=0$, we have $\nabla^2 E(\theta) {w}={0}$, which implies that $T_{\theta} \mathcal{M} \subseteq \operatorname{Ker} (\nabla^2 E(\theta))$. 
This observation implies that the nontrivial behavior of the Hessian is restricted to the normal space $N_{\theta} \mathcal{M}$.  
Consequently, the concept of nondegeneracy can be naturally extended to the critical submanifold $\mathcal{M}$.

\begin{definition}[nondegenerate critical submanifold]
A smooth submanifold $\mathcal{M}$ $\hookrightarrow \mathbb{R}^d$ is called a \emph{nondegenerate} critical submanifold of $E$ if \cite{nicolaescu2007invitation}:
\begin{enumerate}
\item[(1)] $\mathcal{M}\subset \mathcal{M}_E$;
\item[(2)] $\mathcal{M}$ is connected;
\item[(3)] $\forall \theta\in \mathcal{M}$, $\nabla^2 E(\theta)$ is nondegenerate on $N_{\theta} \mathcal{M}$.
\end{enumerate}
\end{definition}
Note that the condition (3) is equivalent to $T_{\theta} \mathcal{M}=\operatorname{Ker} (\nabla^2 E(\theta))$.
Finally, we introduce the definition of Morse--Bott functions.
\begin{definition}[Morse--Bott function]
A smooth function $E:\mathbb{R}^d\to\mathbb{R}$ is called a \emph{Morse--Bott function} if every connected component of $\mathcal{M}_E$ is a nondegenerate critical submanifold.
\end{definition}

For Morse--Bott functions, the following lemma gives a fundamental result that provides a local normal form around a critical point on a critical submanifold \cite{Banyaga2004}.
\begin{lemma}[Morse--Bott lemma]
\label{lemma:bott}
Let $E: \mathbb{R}^d \to \mathbb{R}$ be a Morse--Bott function, and let $\mathcal{M}$ be a connected component of $\mathcal{M}_E$ of dimension $m$. 
For any index-$s$ saddle point $\theta^* \in \mathcal{M}$, there exists a neighborhood $U\subset \mathbb{R}^d $ of $\theta^*$ and a local diffeomorphism $\phi: U \to V \subset \mathbb{R}^m \times \mathbb{R}^{d-m}$ such that $\phi(\theta^*) = 0$.
If we denote the local coordinates by $(x, y) = (x_1, \ldots, x_m, y_1, \ldots, y_{d-m})$, then $\phi$ straightens the manifold locally such that $\phi(U \cap \mathcal{M}) = \{(x, y) \in V : y = 0\}$. 
Furthermore, for all $\theta\in U$, the function $E$ can be expressed in these coordinates $\phi(\theta)=(x,y)$ as:
\begin{equation}\label{eqn:morsebott}
E(\theta) = E(\mathcal{M}) - y_1^2 - \cdots - y_s^2 + y_{s+1}^2 + \cdots + y_{d-m}^2.
\end{equation}
\end{lemma}
This lemma shows that near any critical submanifold, a Morse--Bott function remains constant along the critical submanifold and behaves like a nondegenerate quadratic form in the normal directions.

\subsection{Review of the HiSD Method}
Recently, the HiSD method has emerged as a potent tool for searching for saddle points of any index \cite{yin2019high}. 
Assumed that the energy function is $E\in \mathcal{C}^2(\mathbb{R}^d, \mathbb{R})$,
the HiSD for an index-$k$ saddle point ($k$-HiSD) of the energy function $E$ has the form of:
\begin{equation}\label{eqn:hisd}
\left\{
\begin{aligned}
\dot{\theta} & = - \beta
\left( I - \sum_{p=1}^k 2v_p v_p^\top \right)
\nabla E(\theta),\\
\dot{v}_p &=-\zeta
\left(I-v_pv_p^\top-\sum_{q=1}^{p-1}2v_qv_q^\top\right)
\nabla^2E(\theta)v_p, \quad p=1,\cdots,k,
\end{aligned}
\right.
\end{equation}
where $\{v_p\}_{p=1}^k$ are orthonormal vectors which approximate eigenvectors corresponding to the smallest $k$ eigenvalues of the Hessian $\nabla^2 E(\theta)$.
$\beta$ and $\zeta$ are positive relaxation parameters.
$v^\top w$ represents the inner product of $v$ and $w$ in the Euclidean space $\mathbb{R}^d$.

This dynamical system \eqref{eqn:hisd} originates from the minimax variational characterization of a nondegenerate index-$k$ saddle point and the construction of its associated $k$-dimensional unstable subspace. 
With gradient ascent flow on the subspace spanned by $\{v_p\}_{p=1}^k$ which approximate the smallest $k$ eigenvectors of $\nabla^2E(\theta)$ and gradient descent flow on its orthogonal complement, this dynamics guides the search for an index-$k$ saddle point.
With an initial condition satisfying that $\{v_p\}_{p=1}^k$ are orthonormal, the solution to \eqref{eqn:hisd} satisfies this orthonormal condition. 

\subsection{Numerical algorithms of HiSD}
In discrete numerical algorithms, we denote by $\theta^{(t)}\in \mathbb{R}^{d}$ the parameter values after $t$ iteration steps, and by $v_p^{(t)}$ the directions. 
An explicit Euler scheme of the HiSD method is presented in \Cref{alg:ori}.
Here, ``EigenSol'' represents an eigenvector solver that, given initial values $\{ v_p^{(t)}\}_{p=1}^k$, computes eigenvectors corresponding to the smallest $k$  eigenvalues of $\nabla^2 E(\theta^{(t+1)})$, \\---either accurately or approximately. Typical choices include simultaneous Rayleigh-quotient iterative minimization methods  \cite{longsine1980simultaneous}, and locally optimal block preconditioned conjugate gradient (LOBPCG) methods \cite{knyazev2001toward}.

\begin{algorithm}[H]
\caption{Numerical algorithm of $k$-HiSD}
\begin{algorithmic}
\REQUIRE{$k\in\mathbb{N}$, $\theta^{(0)}\in\mathbb{R}^d$, 
$\left\{ v_p^{(0)}\right\}_{p=1}^k\subset \mathbb{R}^d$ satisfying 
${v}_{p}^{(0){\top}} {{ v_q}}^{(0)}=\delta_{pq}$.}
\FOR{$t=0,1,\cdots,T-1$}
\STATE {$\theta^{(t+1)} = \theta^{(t)} - \beta_t
\left(I-\sum_{p=1}^k 2 v_p^{(t)}{v}_{p}^{(t){\top}}\right)
\nabla E(\theta^{(t)})$.}
\hfill{(main iteration)}
\STATE {$\left\{ v_p^{(t+1)}\right\}_{p=1}^k = \operatorname{EigenSol}\left( \left\{ v_p^{(t)}\right\}_{p=1}^k, \nabla^2 E(\theta^{(t+1)}) \right)$.}
\ENDFOR
\ENSURE $\theta^{(T)}$
\end{algorithmic}
\label{alg:ori}
\end{algorithm}

In practical computations, the explicit Euler scheme of HiSD often suffers from slow convergence, thereby requiring an excessive number of iterations to compute saddle points. 
To address this issue, momentum‐accelerated variants of HiSD have been developed. 
Inspired by the simplicity and acceleration efficiency of the Heavy‐Ball technique \cite{Polyak1964SomeMO}, the HiSD method with Heavy‐Ball acceleration is proposed in \cite{luo2025accelerated}. 
By setting $\theta^{(-1)}=\theta^{(0)}$, the main iteration step in \Cref{alg:ori} is given by:
\begin{equation}\label{eqn:neavyball}
\theta^{(t+1)} = \theta^{(t)} - \beta_t
\left(I-\sum_{p=1}^k 2 v_p^{(t)}{v}_{p}^{(t){\top}}\right)
\nabla E(\theta^{(t)})+ \gamma (\theta^{(t)}-\theta^{(t-1)}),  
\end{equation}
where $\gamma \in [0, 1)$ is the momentum parameter. 

Similarly, Nesterov's acceleration \cite{Nesterov1983} has been applied to improve the performance of the HiSD method.
Specifically, a variant incorporating this technique has been implemented in \cite{liu2024neural}, with the main iteration step in \Cref{alg:ori} as follows:
\begin{equation*}
\xi^{(t)} =\theta^{(t)}+\gamma_t \left(\theta^{(t)}-\theta^{(t-1)}\right),
\quad
\theta^{(t+1)} = \xi^{(t)} - \beta_t
\left(I-\sum_{p=1}^k 2 v_p^{(t)}{v}_{p}^{(t){\top}}\right)
\nabla E(\xi^{(t)}),
\end{equation*}
where $\gamma_t$ is the momentum parameter, typically set to $\gamma_t=\frac{t}{t+3}$. 
In practical computations, restart strategies are often employed to further enhance performance.

\section{Analysis of HiSD in degenerate cases}\label{sec:3}
In this section, we present the theoretical analysis of HiSD for degenerate saddle points on a critical manifold. 

First, based on the local normal form in \Cref{lemma:bott}, we illustrate our idea with a simple example. 
Consider $\theta = (x,y) = (x_1, \cdots, x_m, y_1, \cdots, y_{d-m}) \in\mathbb{R}^d$ and the canonical form of the energy function,
\begin{equation*}
E(\theta) = - y_1^2 - \cdots - y_s^2 + y_{s+1}^2 + \cdots + y_{d-m}^2.
\end{equation*}
In this example, $\mathcal{M} = \{(x,0)\in \mathbb{R}^d\}$ constitutes an $m$-dimensional nondegenerate critical manifold of index-$s$ saddle points, and $U=\{(x,0)\in \mathbb{R}^d\}$ also serves as the nullspace of the Hessian at any point. 
For the $k$-HiSD with $k=s,\cdots,s+m$, we take the first $s$ unstable directions as the eigenvectors $v_p = (0,e_p)$, $p=1,\cdots,s$, where $e_p$ is the standard basis vector, and the $(k-s)$ other unstable directions $v_p$, $p=s+1, \cdots,k$ as an arbitrary orthonormal set in $U$.
Then the explicit Euler scheme with fixed step sizes $\beta_t\equiv\beta\in(0,1)$,
\begin{equation*}
\theta^{(t+1)} = \theta^{(t)} - \beta\left(0, 2y^{(t)}\right) = \left(x^{(t)}, (1-2\beta)y^{(t)}\right),
\end{equation*}
converges to a saddle point $\left(x^{(0)}, 0\right)\in\mathcal{M}$ at a linear rate. 
This example intuitively illustrates the flexibility of the index $k$ of HiSD and the $(k-s)$ unstable direction in the degenerate setting. 

We now proceed to the rigorous theoretical analysis. 
We assume that $\mathcal{M}\subset\mathbb{R}^d$ is a closed, $\mathcal{C}^2$–embedded, connected, nondegenerate critical manifold of dimension $m$, consisting of index-$s$ saddle points of a $\mathcal{C}^2$ function $E$ with bounded Hessian. 
Under this setting, we investigate the convergence of the $k$-HiSD to the saddle manifold $\mathcal{M}$.
Specifically, we establish the local convergence for $k$-HiSD with $k=s,\dots,s+m$ and derive the corresponding convergence rate.

\subsection{Notations}
For a subspace $V\subset\mathbb{R}^{d}$, the symbol $V^\perp$ denotes its orthogonal complement, and $\mathcal{P}_V: \mathbb{R}^{d}\to \mathbb{R}^{d}$ denotes the orthogonal projection operator onto $V$.
Similarly, we define the nearest-point projection operator $\mathcal{P}_{\mathcal{M}}$,
\begin{equation}\label{eqn:projectionm}
\mathcal{P}_\mathcal{M}(\theta) =
\argmin_{\theta'\in\mathcal M} 
\|\theta-\theta'\|_2,
\end{equation}
to project $\theta \in \mathbb{R}^{d}$ onto the saddle manifold $\mathcal{M}$ and we employ $\hat{\theta}$ to denote $\mathcal{P}_\mathcal{M}(\theta)$ for brevity.
The nearest-point projection satisfies the important property that
\begin{equation}\label{eqn:proj}
\theta-\hat{\theta} \in 
\operatorname{Im}( \nabla^2 E(\hat{\theta}))
=\left(\operatorname{Ker}( \nabla^2 E(\hat{\theta}))\right)^\perp. 
\end{equation}

\begin{remark}
The nearest-point projection operator $\mathcal{P}_\mathcal{M}$ is well defined at least in a neighborhood of $\mathcal{M}$.
\begin{enumerate}
\item \emph{Existence.}
Because $\mathcal M$ is closed in $\mathbb R^d$, the infimum in the definition
\begin{equation}\label{eqn:distance}
\operatorname{dist}(\theta,\mathcal M) =
\inf_{\theta'\in\mathcal M} 
\|\theta-\theta'\|_2,
\end{equation}
can be attained for every $\theta \in \mathbb R^d$.
\item \emph{Uniqueness.}
By the tubular‐neighborhood theorem \cite{hirsch2012differential}, 
for each ${\vartheta}\in\mathcal M$, there exists $r_{{\vartheta}}>0$ such that the normal bundle map
\[
  (\theta',{v})\mapsto \theta'+{v},\quad
  \theta'\in\mathcal M,\;{v}\in N_{\theta'}\mathcal M,\;\|{v}\|_2
  < r_{{\vartheta}},
\]
is a diffeomorphism onto the open tube
\[
  \mathcal{U}_{{\vartheta}}=
  \left\{\theta\in\mathbb R^d:\operatorname{dist}(\theta,\mathcal M)<r_{{\vartheta}}, \|\theta-{\vartheta}\|_2 <r_{{\vartheta}}\right\}.
\]
Consequently, for each $\theta\in \mathcal{U}_{{\vartheta}}$, the nearest‐point projection  $\mathcal P_{\mathcal M}(\theta)$ is unique.
\end{enumerate}
\end{remark}

\begin{remark}
$\hat\theta$ is a local minimizer of the constrained problem $\min\limits_{\theta'\in\mathcal{M}}\;\tfrac12\|\theta-\theta'\|_2^2$, 
and the KKT condition 
$\theta-\hat\theta + \nabla^2 E(\hat\theta)\lambda = 0$
yields $\theta - \hat\theta \in \operatorname{Im}(\nabla^2 E(\hat\theta))$.
The second equality in \eqref{eqn:proj} follows from the symmetry of $\nabla^2 E(\hat\theta)$.
\end{remark}

The following assumption, which holds throughout this section, is proposed in a suitable tubular neighborhood of the saddle manifold $\mathcal{M}$.
Denote the eigenvalues of $\nabla^2 E(\theta)$ by $\{\lambda_p\}_{p=1}^{d}$ arranged in ascending order and the corresponding orthonormal eigenvectors by $\{u_p\}_{p=1}^d$, and those of $\nabla^2E(\hat{\theta})$ by $\{\hat{\lambda}_p\}_{p=1}^{d}$ and $\{\hat{u}_p\}_{p=1}^d$, respectively.
\begin{assumption}\label{ass:1}
There exist a $\delta$-tubular neighborhood $\mathcal{U}(\mathcal{M},\delta)=\{\theta\in \mathbb{R}^d:\|\theta-\hat{\theta}\|_2 < \delta\}$ and positive constants $M,L,\mu>0$ satisfying $M\delta \leqslant {\mu}/{4}$, such that:
\begin{enumerate}
\item [(i)] The nearest-point projection operator $\mathcal{P}_{\mathcal{M}}$ is well-defined in $\mathcal{U}(\mathcal{M},\delta)$.
\item [(ii)] For all $\theta, \theta' \in \mathcal{U}(\mathcal{M},\delta)$, $\|\nabla^2 E(\theta)-\nabla^2 E(\theta')\|_2\leqslant M\|\theta-\theta'\|_2$.
\item [(iii)] For all $\theta\in \mathcal{U}(\mathcal{M},\delta)$, 
$\lambda_1, \cdots, \lambda_s \in [-L, -\mu]$, 
$\lambda_{s+m+1}, \cdots, \lambda_{d} \in [\mu,L]$.
\end{enumerate}
\end{assumption}

\begin{remark}
For all $\theta\in \mathcal{U}(\mathcal{M},\delta)$,
Assumption \ref{ass:1} implies an upper bound for $|\lambda_p|$ for $p=s+1,\cdots,s+m$. 
Because $\hat{\theta}\in \mathcal{M}$, we directly have $\hat{\lambda}_{p}=0$, 
and hence, based on \cite[Theorem 8.1.5]{golub2013matrix}, we have,
\begin{equation*}
|\lambda_p|=|\lambda_p-\hat{\lambda}_p| \leqslant 
\|\nabla^2 E(\theta)-\nabla^2 E(\hat{\theta})\|_2 \leqslant 
M\|\theta-\hat{\theta}\|_2 \leqslant 
M\delta\leqslant {\mu}/{4}. 
\end{equation*}
Therefore, Assumption \ref{ass:1} ensures that those eigenvalues close to zero are well separated from the others in the neighborhood of $\mathcal{M}$. 
\label{re:asm}
\end{remark}

\subsection{Local convergence of continuous HiSD}
For nondegenerate index-$k$ saddle points, the stability of $k$-HiSD can be directly established through linearization \cite{yin2019high}, while such approach is no longer applicable to saddle manifolds.
In the degenerate cases, we aim to show that $G(\theta)=\frac{1}{2}\|\nabla E(\theta)\|_2^2$ serves as a local Lyapunov function for $k$-HiSD near the saddle manifold $\mathcal{M}$. 

From Assumption \ref{ass:1}, we define $m$-dimensional subspaces $U = \operatorname{span}\{u_p\}_{p=s+1}^{s+m}$ and $\widehat{U} = \operatorname{span}\{\hat{u}_p\}_{p=s+1}^{s+m}$.
The distance between $U$ and $\widehat{U}$, defined as \cite{golub2013matrix}:
\begin{equation}\label{eqn:projdist}
\operatorname{dist}(U, {\widehat{U}}):=
\left\|\mathcal{P}_{U} - \mathcal{P}_{\widehat{U}}\right\|_2= 
\left\|\mathcal{P}_{{U^\perp}} - \mathcal{P}_{\widehat{U}^\perp}\right\|_2,
\end{equation}
can be bounded using the following lemma with $C=2M/\mu$. 
\begin{lemma}[distance between subspaces]
\label{lemma:dist}
For any $\theta \in \mathcal{U}(\mathcal{M}, \delta)$, we have the following estimate: $\operatorname{dist}(U, \widehat{U}) \leqslant C \|\theta-\hat{\theta}\|_2$. 
\end{lemma}
The gradient can also be characterized using the following lemma.
\begin{lemma} \label{lemma:expansion}
For any $\theta, \theta' \in \mathbb{R}^{d}$, the following identity holds:
\begin{align*}
\nabla E(\theta')&=\nabla E(\theta)+\nabla^{2}E(\theta)(\theta'-\theta)+J(\theta'-\theta),\\
\text{where }
J&= \int_0^1 \nabla^2 E\left(\theta+t(\theta'-\theta)\right)\mathrm{d}t - \nabla^2 E(\theta) \in \mathbb{R}^{d\times d}.
\end{align*}
Moreover, if $\theta, \theta' \in \mathcal{U}(\mathcal{M},\delta)$, we have $\|J\|_2\leqslant \frac{M}{2} \|\theta'-\theta\|_2$.
\end{lemma}
\noindent Furthermore, the gradient $\nabla E(\theta)$ and its orthogonal projections onto $U$ and $U^\perp$ satisfy sharp estimate:  
$\mathcal{P}_{U^\perp}\nabla E(\theta)$ dominates and grows linearly with the distance to the manifold, while $\mathcal{P}_{U}\nabla E(\theta)$ is a higher‐order term.
\begin{lemma}[order‐separated gradient decomposition]\label{lemma:decom}
For $\theta \in \mathcal{U}(\mathcal{M},\delta)$, we have 
$\|\nabla E(\theta)\|_2\leqslant L\|\theta-\hat\theta\|_2$ and
\begin{equation}\label{eqn:decom}
\|\mathcal{P}_{U^\perp}\nabla E(\theta)\|_2 \geqslant
\left(\mu-\tfrac{5}{2}M\|\theta-\hat\theta\|_2\right)
\|\theta-\hat\theta\|_2,
\quad
\|\mathcal{P}_{U}\nabla E(\theta)\|_2 \leqslant
M\|\theta-\hat\theta\|_2^2.
\end{equation}
\end{lemma}
\begin{remark}\label{rmk:unique}
From \eqref{eqn:decom}, we have 
$\|\nabla E(\theta)\|_2\geqslant
\left(\mu-\tfrac{5}{2}M\|\theta-\hat\theta\|_2\right)
\|\theta-\hat\theta\|_2
\geqslant 3\mu\|\theta-\hat\theta\|_2/8$ because of $M\delta\leqslant \mu/4$, so
any critical point in the tubular neighborhood $\mathcal{U}(\mathcal{M},\delta)$ lies in the critical manifold $\mathcal{M}$. 
\end{remark}

Our first result establishes a pointwise descent property: $k$-HiSD equipped with some accurate eigenvectors yields a decrease in the Lyapunov function $G$.
\begin{theorem}[pointwise descent]\label{thm:direction}
Under Assumption \ref{ass:1}, for $s\leqslant k\leqslant s+m$, if $\theta \in \mathcal{U}(\mathcal{M},\delta) \backslash \mathcal{M}$, then 
for any orthonormal vectors $\{w_p\}_{p=s+1}^{k}$ in $U$,
the direction 
\begin{equation}
    -\left(I-\sum_{p=1}^s 2 u_p u_p^\top
    -\sum_{p=s+1}^k 2 w_p w_p^\top
    \right) \nabla E(\theta),
\end{equation}
is a strictly decreasing-direction for $G(\theta)=\frac{1}{2}\|\nabla E(\theta)\|_2^2$.
\end{theorem}
\begin{proof}
Denote $g = \nabla E(\theta)$, and it suffices to prove
\begin{equation}\label{eqn:ktheta}
K(\theta):=-g^{\top}\nabla^{2}E(\theta)\left(I-\sum_{p=1}^s 2u_p u_p^\top -\sum_{p=s+1}^k 2w_p w_p^\top \right) g< 0.
\end{equation}
Denote $W=\operatorname{span}\{w_p\}_{p=s+1}^{k}\subseteq U$ and $\mathcal{P}_{W}=\sum_{p=s+1}^k w_p w_p^\top$ is the orthogonal projection onto $W$.
From the eigenvalue decomposition $\nabla^2 E(\theta) = \sum_{p=1}^{d} \lambda_p u_p u_p^\top$ and the eigenvalue bound in Assumption \ref{ass:1}, we obtain
\begin{equation*}
\begin{aligned}
K(\theta) 
&=\sum_{p=1}^{s}\lambda_{p}(u_{p}^{\top}g)^2-
\sum_{p=s+1}^{d}\lambda_{p}(u_{p}^{\top}g)^2
+2g^{\top}\nabla^{2}E(\theta)\mathcal{P}_{W}g \\
&\leqslant -\mu 
\left(\sum_{p=1}^{s}(v_{p}^{\top} g)^2
+\sum_{p=s+m+1}^{d}(v_{p}^{\top}g)^2\right)+
\frac{\mu}{4}\left(\sum_{p=s+1}^{s+m}(v_{p}^{\top}g)^2+
2 \|g\|_2\|\mathcal{P}_{W}g\|_2 \right)\\
&\leqslant -\mu \|\mathcal{P}_{U^\perp}g\|_2^2
+ \frac{\mu}{4}
\left(\|\mathcal{P}_{U}g\|_2^2+
\|g\|_2^2 + \|\mathcal{P}_{U}g\|_2^2 \right)
= -\frac{3\mu}{4}
\left(\|\mathcal{P}_{U^\perp}g\|_2^2 
- \|\mathcal{P}_{U}g\|_2^2\right).
\end{aligned}
\end{equation*}
By applying \Cref{lemma:decom}, we have
\begin{equation}
\|\mathcal{P}_{U^\perp}g\|_2 \geqslant
\left(\mu-\tfrac{5}{2}M\|\theta-\hat\theta\|_2\right)
\|\theta-\hat\theta\|_2,
\quad
\|\mathcal{P}_{U}g\|_2 \leqslant
M\|\theta-\hat\theta\|_2^2.
\end{equation}
Therefore, from $M\|\theta-\hat\theta\|_2<M\delta\leqslant \mu/4$, we have
\begin{equation}\label{eqn:kest}
\begin{aligned}
K(\theta) & = -\frac{3\mu}{4}
\left(\|\mathcal{P}_{U^\perp}g\|_2 + \|\mathcal{P}_{U}g\|_2\right)
\left(\|\mathcal{P}_{U^\perp}g\|_2 - \|\mathcal{P}_{U}g\|_2\right) \\
&\leqslant -\frac{3\mu}{4}\|g\|_2 
\left(\mu-\frac{7}{2}M\|\theta-\hat\theta\|_2\right)
\|\theta-\hat\theta\|_2 
\leqslant -\frac{3\mu^2}{32}\|g\|_2 
\|\theta-\hat\theta\|_2 <0,
\end{aligned}
\end{equation}
which completes the proof. 
\end{proof}

\begin{remark}
Because of the possible presence of multiple eigenvalues of $\nabla^2 E(\theta)$, the choice of eigenvector direction may not be unique, so the trajectory of $k$-HiSD with accurate eigenvectors may not be well defined. 
For $\theta \in \mathcal{U}(\mathcal{M},\delta)$, Remark \ref{re:asm} shows that the eigenvalues $\{\lambda_p\}_{p=s+1}^{s+m}$ are close to zero, and \Cref{lemma:decom} indicates that the orthogonal projection of $\nabla E(\theta)$ onto $U$ is sufficiently small. 
Consequently, in \Cref{thm:direction}, we do not explicitly require the vectors $w_p$ as eigenvectors, but instead assume that they are orthonormal vectors within the subspace $U$. 
\end{remark}

The following result establishes the local asymptotic stability and convergence of $k$-HiSD \eqref{eqn:hisd}, provided that the first $s$ unstable directions are sufficiently accurate while the other $(k-s)$ unstable directions are allowed to vary near the subspace $U(t)$.
Here, we consider the trajectory $\theta$ depending on $t$, and maintain the previous notations for $\lambda_p$, $u_p$, $U$, which also depend on $t$.
\begin{theorem}[local asymptotic stability and trajectory convergence]\label{thm:inaccurate}
Under Assumption \ref{ass:1}, for $s\leqslant k\leqslant s+m$, assume that the trajectory $\theta(t)$ generated by the $k$-HiSD \eqref{eqn:hisd} remains within $\mathcal{U}(\mathcal{M},\delta)$ for all $t\geqslant 0$. 
If there exists a constant $\alpha< \frac{3\mu^2}{64L^2}$ such that for any $t \geqslant 0$, the unstable directions $v_p(t)$ in $k$-HiSD satisfy
\begin{equation}
\left\| \sum_{p=1}^k v_p(t)v_p(t)^\top  
- \sum_{p=1}^s u_p(t)u_p(t)^\top
- \sum_{p=s+1}^k w_p(t)w_p(t)^\top
\right\|_2 < \alpha, 
\end{equation}
where $\{w_p{(t)}\}_{p=s+1}^{k}$ are some orthonormal vectors in $U(t)$, 
then $\mathcal{M}$ is asymptotically stable and the trajectory $\theta(t)$ converges to an index-$s$ saddle point in $\mathcal{M}$.
\end{theorem}
\begin{proof}
According to \cite[Theorem 4.18]{bhatia2002}, to prove $\mathcal{M}$ is asymptotically stable, it suffices to prove that $G(\theta)=\frac{1}{2}\|\nabla E(\theta)\|_2^2$ serves as a strict Lyapunov function for the $k$-HiSD \eqref{eqn:hisd}:
\begin{enumerate}
\item [(i)] For $\vartheta\in \mathcal{M}$, $G(\vartheta) = 0$. 
 For $\vartheta \in \mathcal{U}(\mathcal{M},\delta) \backslash \mathcal{M}$, $G(\vartheta) > 0$. 
\item [(ii)] There exist strictly increasing functions $a(x), b(x), a(0)= b(0)=0$, defined for $x \geqslant 0$, such that
$a(\|\vartheta-\hat\vartheta\|_2) \leqslant G(\vartheta) \leqslant b(\|\vartheta-\hat\vartheta\|_2)$ for $\vartheta \in \mathcal{U}(\mathcal{M},\delta)$.
\item [(iii)] If $\theta(0) \notin \mathcal{M}$, then along the $k$-HiSD trajectory we have $\frac{\mathrm{d}}{\mathrm{d}t}G(\theta(t)) <0$ and $\lim_{t\to \infty} G(\theta(t)) =0$.
\end{enumerate} 
We have $\|\nabla E(\vartheta)\|_2\geqslant 3\mu\|\vartheta-\hat\vartheta\|_2/8$ from Remark \ref{rmk:unique}, and $\|\nabla E(\vartheta)\|_2 \leqslant L\|\vartheta-\hat\vartheta\|_2$ from \Cref{lemma:decom}, so conditions (i) and (ii) hold for $a(x)=9\mu^2x^2/128$ and $b(x)=L^2x^2/2$.

Next, to show (iii), we omit the explicit dependence on $t$ for simplicity in the following derivations.
Denote $g = \nabla E(\theta)$, $K(\theta)$ as \eqref{eqn:ktheta}, and 
\begin{equation}
R=\sum_{p=1}^k v_pv_p^\top  
- \sum_{p=1}^s u_pu_p^\top
- \sum_{p=s+1}^k w_pw_p^\top.    
\end{equation}
From the $\theta$-dynamics in $k$-HiSD,
\begin{equation}
\dfrac{\mathrm{d}\theta}{\mathrm{d}t} 
=-\left(I-\sum_{p=1}^k 2v_p v_p^\top\right)g
=-\left(I-\sum_{p=1}^s 2u_p u_p^\top- \sum_{p=s+1}^k 2w_p w_p^\top-2 R\right)g,
\end{equation}
we define a constant $\eta=\frac{3\mu^2}{32L}-2L\alpha>0$ and obtain
\begin{equation*}
\label{eqn:lyp}
\begin{aligned}
\dfrac{\mathrm{d}}{\mathrm{d}t}G(\theta)&=
g^{\top}\nabla^{2}E(\theta)\dfrac{\mathrm{d}\theta}{\mathrm{d}t}=K(\theta)+2g^{\top}\nabla^{2}E(\theta)Rg
\leqslant 
K(\theta)+2\|g\|_2^2
\left\|\nabla^{2}E(\theta)\right\|_2 \|R\|_2
\\&
\leqslant 
-\left(\frac{3\mu^2}{32}-2L^2\alpha\right)\|g\|_2\|\theta-\hat\theta\|_2
\leqslant
-2\eta G(\theta),
\end{aligned}
\end{equation*}
where we apply \eqref{eqn:kest}, 
$\|\nabla^{2}E(\theta)\|_2\leqslant L$,
and $\|g\|_2\leqslant L\|\theta-\hat\theta\|_2$ from \Cref{lemma:decom}.

From the Gronwall's inequality, we directly obtain 
$G(\theta(t))\leqslant G(\theta(0))\mathrm{e}^{-2\eta t}$,
that is, 
$\|\nabla E(\theta(t))\|_2\leqslant C_0 \mathrm{e}^{-\eta t}$, 
where $C_0=\|\nabla E(\theta(0))\|_2>0$.
Therefore, we have $\lim\limits_{t\to \infty} G (\theta(t))=0$, which completes the proof of (iii).

Finally, from $\left\|\frac{\mathrm{d}\theta}{\mathrm{d}t}(t)\right\|_2 = \|\nabla E(\theta(t))\|_2$, for any $t\geqslant0$, $\tau>0$, we have
\begin{equation*}
\| \theta(t+\tau) - \theta(t)\|_2 = 
\left\|\int_t^{t+\tau} \frac{\mathrm{d}\theta}{\mathrm{d}t}(s) \mathrm{d}s\right\|_2
\leqslant
\int_t^{+\infty}C_0\mathrm{e}^{-\eta s}\mathrm{d}s
=\frac{C_0}{\eta}\mathrm{e}^{-\eta t},
\end{equation*}
which means that $\{\theta(t)\}_{t\geqslant0}$ is a Cauchy function as $t\to +\infty$, so $\lim\limits_{t\to \infty} \theta(t)$ exists, 
From $\lim\limits_{t\to \infty} \|\nabla E(\theta(t))\|_2=0$ and Remark \ref{rmk:unique}, we have $\lim\limits_{t\to \infty} \theta(t)$ is a critical point in $\mathcal{M}$. 
\end{proof}

\subsection{Local convergence of discrete HiSD algorithms}
Under the assumption of nondegeneracy, 
it was theoretically established in \cite{luo2022convergence} that the explicit Euler scheme of the HiSD method with fixed step sizes exhibits a linear convergence rate in a local domain.
For degenerate saddle points, the primary challenge in establishing local convergence lies in their inherent degeneracy, which prevents the direct application of the techniques developed in \cite{luo2022convergence} and thus necessitates a tailored analysis. 
To address this, we extend the numerical analysis in \cite{luo2022convergence} by showing that, under suitable assumptions, the algorithm can converge to the saddle manifold at a linear convergence rate. 

First, we introduce a characterization of $\theta^{(t)}-\hat{\theta}^{(t)}$, the proof of which is very similar to \cite[Lemma 3.2]{luo2022convergence}.
\begin{lemma}[error of one-step iteration]
\label{lemma:one_step_estimate}
For an iteration scheme 
$\theta^{(t+1)} = \theta^{(t)} - \beta_t A^{(t)}\nabla E(\theta^{(t)})$ 
with $A^{(t)}\in\mathbb{R}^{n\times n}$, we have the following identity,
\begin{equation}
\label{eqn:recursion}
\begin{aligned}
&\theta^{(t+1)}-\hat{\theta}^{(t)} = 
\left[Q^{(t)} + B^{(t)}\right] (\theta^{(t)}-\hat{\theta}^{(t)})\\
\text{where }
&Q^{(t)} = I-\beta_t  A^{(t)}\nabla^2 E(\theta^{(t)}),\\
&B^{(t)} = \beta_t A^{(t)}\left[ \nabla^2 E(\theta^{(t)})
- \int_0^1 \nabla^2 E\left(\hat{\theta}^{(t)}+t(\theta^{(t)}-\hat{\theta}^{(t)})\right) \mathrm{d}t\right].
\end{aligned}
\end{equation}
Furthermore, if $\theta^{(t)}\in \mathcal{U}(\mathcal{M,\delta})$, then we have
$\|B^{(t)}\|_2\leqslant
\frac{1}{2}\beta_tM\|A^{(t)}\|_2\|\theta^{(t)}-\hat{\theta}^{(t)}\|_2$.
\end{lemma}

Specially, for the following explicit Euler scheme of $k$-HiSD:
\begin{equation}\label{exact_position_k}
\theta^{(t+1)}=\theta^{(t)}-\beta_t
\left(I-2\sum_{p=1}^kv_p^{(t)}{v_p^{(t)}}^{\top}\right)
\nabla E(\theta^{(t)}),
\end{equation}
with ${{v_p^{(t)\top}}}{v_q^{(t)}} = \delta_{pq}$, 
denote the eigenvalues and the corresponding orthonormal eigenvectors of $\nabla^2 E(\theta^{(t)})$ and $\nabla^2 E(\hat{\theta}^{(t)})$ by $\left(\lambda_p^{(t)}, u_p^{(t)}\right)_{p=1}^d$ and $\left(\hat{\lambda}_p^{(t)}, \hat{u}_p^{(t)}\right)_{p=1}^d$ arranged in ascending order, respectively. 
For $\theta^{(t)} \in \mathcal{U}(\mathcal{M},\delta)$, we can further define $U^{(t)} = \operatorname{span}\{u_p^{(t)}\}_{p=s+1}^{s+m}$ and 
$\widehat{U}^{(t)} = \operatorname{span}\{\hat{u}_p^{(t)}\}_{p=s+1}^{s+m}$.
Now we present the following single-step numerical analysis with the accurate eigenvectors.
\begin{theorem}[error estimate in one-step iteration]
\label{thm:one_step}
Under Assumption \ref{ass:1}, assume that the index of $k$-HiSD is chosen as $s\leqslant k\leqslant s+m$, $\theta^{(t)} \in \mathcal{U}(\mathcal{M},\delta)$, and $\beta_t \leqslant 2/(L+\mu)$.
If $v_p^{(t)}=u_p^{(t)}$ for $1\leqslant p\leqslant s$ 
and $v_p^{(t)} \in U^{(t)}$ for $s+1\leqslant p\leqslant k$, 
then for $r_t = \|\theta^{(t)}-\hat{\theta}^{(t)}\|_2$, the following estimate holds:
\begin{equation}\label{equ:one_step_estimate}
r_{t+1}\leqslant
\left(1-\mu\beta_t\right) r_t + \left(C 
+ \frac{5M\beta_t}{2}\right) r_t^2.
\end{equation}
\end{theorem}

\begin{proof}
By (\ref{exact_position_k}), we choose $A^{(t)} = I-2\sum_{p=1}^k v_p^{(t)}{v_{p}^{(t)\top}}$ and apply \Cref{lemma:one_step_estimate} with $\|A\|_{2} = 1$ to get 
\begin{equation}\label{eqn:qbdecom}
\theta^{(t+1)}-\hat{\theta}^{(t)} = 
\left[Q^{(t)} + B^{(t)}\right] 
\left(\theta^{(t)}-\hat{\theta}^{(t)}\right),
\quad
\left\|B^{(t)}\right\|_2\leqslant
\frac{1}{2}M\beta_tr_t.
\end{equation}
With the triangular inequality, we ascertain that
\begin{equation}\label{eqn:rbound1}
\left\|\theta^{(t+1)}-\hat{\theta}^{(t)}\right\|_2\leqslant
\left\|Q^{(t)}  \left(\theta^{(t)}-\hat{\theta}^{(t)}\right)\right\|_2 
+ \frac{M}{2}\beta_tr_t^2,
\end{equation}
and our next step is to estimate $\left\|Q^{(t)}\left(\theta^{(t)}-\hat{\theta}^{(t)}\right)\right\|_2$.

Below, we drop all the subscript $(t)$ for simplicity, and consider the decomposition of $\theta-\hat{\theta}$ as $\theta-\hat{\theta}=g_1+g_2$, 
where $g_1=\mathcal{P}_{U^\perp}(\theta-\hat{\theta})$ 
and $g_2=\mathcal{P}_{U}(\theta-\hat{\theta})$. 
From \eqref{eqn:proj} and the nondegeneracy of the critical manifold $\mathcal{M}$, we have $\theta-{\hat{\theta}}
\in {\widehat{U}}^\perp$, so $\|g_2\|_2$ can be estimated from \Cref{lemma:dist} as follows:
\begin{equation}
\|g_2\|_2
=\|\mathcal{P}_{U}(\theta-\hat{\theta})-\mathcal{P}_{\widehat{U}}(\theta-\hat{\theta})\|_2
\leqslant \|\theta-\hat{\theta}\|_2 \cdot
\operatorname{dist}(U, \widehat{U})
\leqslant C\|\theta-\hat{\theta}\|_2^2.
\label{eqn:residual}
\end{equation}

From the eigenvalue decomposition of $\nabla^2 E(\theta)= \sum_{p=1}^{d} \lambda_pu_pu_p^{\top}$, 
$Q$ in \eqref{eqn:qbdecom} can be represented as: 
\begin{equation}
\begin{aligned}
Q &= 
I-\beta_t \left(I-2\sum_{p=1}^k v_pv_p^{\top}\right)\left(\sum_{p=1}^{d} \lambda_pu_pu_p^{\top}\right)\\
&=
\underbrace{\sum_{p=1}^s(1+\beta_t\lambda_p)u_pu_p^{\top}
+ \sum_{p=s+m+1}^{d}(1-\beta_t\lambda_p)u_pu_p^{\top}
}_{Q_1}
+
\underbrace{
\sum_{p=s+1}^{s+m}(1-\beta_t\lambda_p)u_pu_p^{\top}
}_{Q_2}
\\&
\quad+
\underbrace{2\beta_t
\left(\sum_{p=s+1}^k v_pv_p^{\top}\right)
\left(\sum_{p=s+1}^{s+m} \lambda_pu_pu_p^{\top}\right)}_{Q_3}.
\label{eqn:Q(t)}
\end{aligned}
\end{equation}
Because of $\theta \in \mathcal{U}(\mathcal{M},\delta)$ and Assumption \ref{ass:1}, we have
\begin{equation*}
-{\lambda_p}\in[\mu,L],\; \text{for }
1\leqslant p\leqslant s;\quad
{\lambda_p}\in[\mu,L],\; \text{for }
s+m+1\leqslant p\leqslant d,
\end{equation*}
and from Remark \ref{re:asm}, we have $|{\lambda_p}|\leqslant \mu/4$ for $s+1\leqslant p\leqslant s+m$.
Consequently, it follows from the expression of $Q$ in \eqref{eqn:Q(t)} that 
\begin{equation}
\begin{aligned}
&\|Q\mathcal{P}_{U^\perp}\|_{2}
=\left\|Q_1\right\|_2
\leqslant 1-\mu\beta_t,\\
&\|Q\|_{2}\leqslant
\|Q_1 + Q_2\|_2 + \|Q_3\|_2 
\leqslant
1+\dfrac{\beta_t\mu}{4} + 2\beta_t \cdot \dfrac{\mu}{4}
<
1+\mu\beta_t.
\end{aligned}
\end{equation}
Building upon the observations set forth, we now stand ready to furnish a conclusive estimate of $\|Q(\theta-\hat{\theta})\|_{2}$:
\begin{equation}\label{eqn:qrbound}
\begin{aligned}
&\|Q(\theta-\hat{\theta})\|_{2}\leqslant
\|Qg_1\|_{2}+ \|Q\|_{2}\|g_2\|_{2} = \|Q\mathcal{P}_{U^\perp}(\theta-\hat{\theta})\|_{2}+\|Q\|_{2}\|g_2\|_{2} \\
&\;\;\leqslant \|Q\mathcal{P}_{U^\perp}\|_{2}\|\theta-\hat{\theta}\|_{2}+\|Q\|_{2}\|g_2\|_{2}
\leqslant\left(1-\mu\beta_t\right)r_t+
\left(1+\mu\beta_t\right)Cr_t^2.
\end{aligned}
\end{equation}
Combining \eqref{eqn:rbound1} and \eqref{eqn:qrbound}, we finally obtain
\begin{equation}\label{eqn:rt+1}
r_{t+1} \leqslant
\left\|\theta^{(t+1)}-\hat{\theta}^{(t)}\right\|_2
\leqslant
\left(1-\mu\beta_t\right)r_t+
\left(1+\mu\beta_t\right)Cr_t^2 + 
\frac{M}{2}\beta_tr_t^2,
\end{equation}
which concludes the proof with $C = 2M/\mu$.
\end{proof}

Finally, to establish a linear convergence rate from \Cref{thm:one_step}, we introduce a lemma from \cite[Lemma 3.2]{luo2022convergence}.
\begin{lemma}
Let $\{r_t\}_{t\geqslant 0}$ be a non-negative series satisfying
\begin{equation*}
r_{t+1}\leqslant (1-q)r_t + cr_t^2,\quad t\geq 0,~~q\in(0,1),~~c>0.
\end{equation*}
\begin{enumerate}
\item[(a)]  If $\displaystyle r_{t}<\frac{q}{c}$ for some $t\geqslant 0$, then 
$\displaystyle r_{t+1}<r_t<\frac{q}{c};$
\item[(b)] If $\displaystyle r_{0}<\frac{q}{c}$, then 
$\displaystyle r_{t+1}\leqslant \left(\frac{1}{1+q}\right)^{t+1}\frac{qr_0}{q-cr_0}$ for all $t\geqslant 0$.
\end{enumerate}
\label{lemma:recursion}
\end{lemma}

Based on \Cref{thm:one_step} and \Cref{lemma:recursion}, we subsequently present the main theorem of convergence rate.
\begin{theorem}[linear convergence rate]
\label{thm:convergence}
Under Assumption \ref{ass:1}, assume that the index of $k$-HiSD is chosen as $s\leqslant k\leqslant s+m$ with 
$v_p^{(t)}=u_p^{(t)}$ for $1\leqslant p\leqslant s$ 
and $v_p^{(t)} \in U^{(t)}$ for $s+1\leqslant p\leqslant k$.
If a fixed step size is chosen as $\beta_t \equiv \beta \leqslant {2}/(L+\mu)$ for any $t\geqslant 0$ and
the initial point $\theta^{(0)} \in \mathcal{U}(\mathcal{M}, \delta)$ satisfies
\begin{equation}
r_0=\left\|\theta^{(0)}-\hat{\theta}^{(0)}\right\|_2<
\hat{r}=\frac{2\mu^2 \beta}{(5\mu\beta+4)M},
\end{equation}
then $r_t=\|\theta^{(t)}-\hat{\theta}^{(t)}\|_2$ monotonically decreases and converges to $0$ as $t\rightarrow \infty$ with at least a linear convergence rate: 
\begin{equation}
r_t\leqslant \left(1+\mu\beta\right)^{-t} \frac{\hat{r}r_0}{\hat{r}-r_0}.
\end{equation}
\end{theorem}
\begin{proof}
From \Cref{thm:one_step} with $t=0$,
because of 
$r_0<\hat{r}$, we have 
$r_{1}<r_0<\hat{r}$ by applying 
\Cref{lemma:recursion} (a) with $q=\mu\beta\in(0,1)$ and $c=q/\hat{r}=C+\frac{5M\beta}{2}>0$.
Inductively, we have $r_{t+1} < r_{t} <\hat{r}$ for any $t \geqslant 0$.
Then we combine \Cref{thm:one_step} for all $t\geqslant 0$ and \Cref{lemma:recursion} (b) to complete the proof.
\end{proof}
\begin{remark}
This result can be generalized to cases with inexact eigenvectors $\{v_p^{(t)}\}_{p=1}^k$, that  satisfy specific accuracy bounds, by applying techniques from \Cref{thm:inaccurate} and \cite{luo2022convergence}. 
\end{remark}

\subsection{Gradient alignment of HiSD}
Similar to gradient descent \cite{forsythe1968asymptotic}, the HiSD algorithm exhibits a tendency  for the gradient $\nabla E(\theta)$ to align with a specific eigenvector of the Hessian $\nabla^2E(\theta)$.
This phenomenon, referred to as ``gradient alignment'', clarifies the asymptotic direction of the gradient and enables a more precise characterization of the algorithm’s convergence rate. 
Naturally, this property also holds for nondegenerate saddle points.
For saddle manifolds, however, the presence of zero eigenvalues introduces significant challenges in proving gradient alignment rigorously. 
To address these challenges, we now introduce additional assumptions and notations that facilitate a precise analysis of gradient alignment in degenerate cases.

\begin{assumption}\label{ass:align}
There exist positive constants $\mu_{1}, \mu_{2}>0$ such that, for all sufficiently large $t$,
\begin{enumerate}
\item [(i)] 
$\mu
<\min\left(|\lambda_s^{(t)}|, |\lambda_{s+m+1}^{(t)}|\right)
<\mu_{1}<\mu_{2}
<\max\left(|\lambda_s^{(t)}|, |\lambda_{s+m+1}^{(t)}|\right)<L$. 
\item [(ii)] 
$\mu
<\min\left(|\hat{\lambda}_s^{(t)}|, |\hat{\lambda}_{s+m+1}^{(t)}|\right)
<\mu_{1}<\mu_{2}<\max\left(|\hat{\lambda}_s^{(t)}|, |\hat{\lambda}_{s+m+1}^{(t)}|\right)<L$.
\end{enumerate}
\end{assumption}
Let $\lambda_{\min}^{(t)}$ denote the eigenvalue with smaller absolute value between 
$\lambda_s^{(t)}$ and $\lambda_{s+m+1}^{(t)}$,
and let $u^{(t)}$ denote the corresponding unit eigenvector of $\nabla^2 E(\theta^{(t)})$.  
Indeed, $u^{(t)}$ is precisely the eigenvector toward which the gradient direction aligns during iterations.
Similarly, for the Hessian $\nabla^2E(\hat{\theta}^{(t)})$, we define $\hat{\lambda}_{\min}^{(t)}$ and ${\hat u}^{(t)}$.
For convenience, we denote by $\mathcal{P}_{u}$ and $\mathcal{P}_{u^\perp}$ the projection operators onto the subspace spanned by the vector $u$ and onto its orthogonal complement, respectively.

To characterize gradient alignment, we define 
$y_{t} =\left\|\mathcal{P}_{u^{(t)}}\nabla E(\theta^{(t)})\right\|_{2}$ as the magnitude of the component of $\nabla E(\theta^{(t)})$ in the direction of the chosen eigenvector,
and $z_{t} =\left\|\mathcal{P}_{u^{(t)\perp}}\nabla E(\theta^{(t)})\right\|_{2}=\sqrt{\left\|\nabla E(\theta^{(t)}\right)\|_2^2-y_t^2}$.
To avoiding excessive technicalities in boundary cases, we impose the following assumption.
\begin{assumption}\label{ass:power}
There exists a constant $\mu_0>0$ such that, for all sufficiently large $t$, $y_{t}\geqslant \mu_{0} r_{t}$, here $r_t = \|\theta^{(t)} - \hat{\theta}^{(t)}\|_2$.
\end{assumption}
The mathematical description of gradient alignment is 
$y_t/\|\nabla E(\theta^{(t)})\|_2 \to 1$, or equivalently, $z_t/y_t \to 0$, and we present the theorem as below to show that $\nabla E(\theta^{(t)})$ becomes asymptotically aligned with the direction $u^{(t)}$:
\begin{theorem}[gradient alignment]
\label{thm:gradient_alignment}
Under Assumptions \ref{ass:1}, \ref{ass:align} and \ref{ass:power}, assume that the index of $k$-HiSD is chosen as $s\leqslant k\leqslant s+m$ with 
$v_p^{(t)}=u_p^{(t)}$ for $1\leqslant p\leqslant s$ 
and $v_p^{(t)} \in U^{(t)}$ for $s+1\leqslant p\leqslant k$.
If a fixed step size is chosen as $\beta_t \equiv \beta \leqslant 1/L $ for any $t\geqslant 0$ and the initial point $\theta^{(0)}\in \mathcal{U}(\mathcal{M},\delta)$ satisfies
\begin{equation}
r_0<\hat{r}=\frac{2\mu^2 \beta}{(5\mu\beta+4)M},
\end{equation}
then 
$z_t/y_t \to 0$ as $t\to\infty$.
\end{theorem}
\begin{proof}
From Assumption \ref{ass:power}, for all sufficiently large $t$, both $u^{(t)}$ and $\hat{u}^{(t)}$ are eigenvectors corresponding to the $p$-th smallest eigenvalue of the Hessian $\nabla^2 E(\theta^{(t)})$ and $\nabla^2 E(\hat{\theta}^{(t)})$, respectively, with the same $p=s$ or $s+m+1$. 
By an argument similar to the proof of \Cref{lemma:dist}, we have
\begin{equation*}
\left\|\mathcal{P}_{u^{(t)}}-\mathcal{P}_{\hat{u}^{(t)}}\right\|_{2}
\leqslant 
C\left\|\theta^{(t)}-\hat{\theta}^{(t)}\right\|_{2},
\quad
\left\|\mathcal{P}_{u^{(t+1)}}-\mathcal{P}_{\hat{u}^{(t)}}\right\|_{2}
\leqslant 
C\left\|\theta^{(t+1)}-\hat{\theta}^{(t)}\right\|_{2}.
\end{equation*}
Since the assumptions of \Cref{thm:convergence} hold, setting $r_t = \|\theta^{(t)} - \hat{\theta}^{(t)}\|_2$ and applying \eqref{eqn:rt+1} yield $\|\theta^{(t+1)} - \hat{\theta}^{(t)}\|_2 < r_t$. Thus,
\begin{equation}\label{eqn:pupu}
\|\mathcal{P}_{u^{(t+1)}}-\mathcal{P}_{u^{(t)}}\|_{2}
\leqslant
\|\mathcal{P}_{u^{(t+1)}}-\mathcal{P}_{\hat{u}^{(t)}}\|_{2}+
\|\mathcal{P}_{u^{(t)}}-\mathcal{P}_{\hat{u}^{(t)}}\|_{2}
\leqslant 2C r_t.
\end{equation}

Applying \Cref{lemma:expansion} at $\theta^{(t)}$ and $\theta^{(t+1)}$ yields
\begin{equation}\label{eqn:t+1} 
\begin{aligned}
&\nabla E(\theta^{(t+1)})
= \nabla E(\theta^{(t)})  + 
\left(\nabla^2E(\theta^{(t)})+ D^{(t)}\right)
\left(\theta^{(t+1)}-\theta^{(t)}\right),
\\
& \left\|D^{(t)}\right\|_2\leqslant 
\frac M2 \left\|\theta^{(t+1)}-\theta^{(t)}\right\|_2
\leqslant M r_t,
\quad 
\left\|\theta^{(t+1)}-\theta^{(t)}\right\|_2
\leqslant 2 r_t.
\end{aligned}
\end{equation}
Using the explicit Euler scheme of $k$-HiSD:
\begin{equation}
\theta^{(t+1)}=
\theta^{(t)}
-\beta R^{(t)}\nabla E(\theta^{(t)}), \quad
\text{where } 
R^{(t)}=I-2\sum_{p=1}^k v_p^{(t)}v_p^{(t)\top},
\end{equation}
we have 
\begin{equation} \label{eqn:power}
\nabla E(\theta^{(t+1)})
= \left(I -
\beta \nabla^2E(\theta^{(t)}) R^{(t)} \right)
\nabla E(\theta^{(t)}) + 
D^{(t)}
\left(\theta^{(t+1)}-\theta^{(t)}\right).
\end{equation}

By direct calculation using the eigen decomposition of $\nabla^2E(\theta^{(t)})$, we obtain
\begin{equation}\label{eqn:expan}
\begin{aligned}
F^{(t)} :=&\; I-\beta\nabla ^2E(\theta^{(t)})R^{(t)}  \\
=&\sum_{p=1}^s   
\left(1+\beta\lambda_p^{(t)}\right)
u_p^{(t)}{u_p^{(t)\top}}
+ \sum_{p=s+1}^{d}
\left(1-\beta\lambda_p^{(t)}\right)
u_p^{(t)}{u_p^{(t)\top}}
\\
&
+2\beta
\left(\sum_{p=s+1}^{s+m} \lambda_pu_p^{(t)}u_p^{(t)\top}\right)
\left(\sum_{p=s+1}^k v_p^{(t)}{v_p^{(t)\top}}\right).
\end{aligned}
\end{equation}
and $\|F^{(t)}\|_2\leqslant 1+\beta \max_{p}\{|\lambda_p^{(t)}|\}\leqslant 2$.

Substituting \eqref{eqn:power} into the definition of \(z_{t+1}\), we obtain
\begin{equation}\label{eqn:332}
\begin{aligned}
z_{t+1}&=\|\mathcal{P}_{{u^{(t+1)}}^\perp}\nabla E(\theta^{(t+1)})\|_{2}\\
&\leqslant
\|\mathcal{P}_{{u^{(t)}}^\perp}\nabla E(\theta^{(t+1)})\|_{2}+
 \|\mathcal{P}_{{u^{(t+1)}}^\perp}-\mathcal{P}_{{u^{(t)}}^\perp}\|_{2} \|\nabla E(\theta^{(t+1)})\|_{2}\\
&\leqslant
\|\mathcal{P}_{{u^{(t)}}^\perp}
F^{(t)}\nabla E(\theta^{(t)})\|_2
+\|D^{(t)}
(\theta^{(t+1)}-\theta^{(t)})\|_{2}+2CLr_{t}^{2}\\
&\leqslant\|\mathcal{P}_{{u^{(t)}}^\perp}
F^{(t)}\nabla E(\theta^{(t)})\|_2+(2M+2CL)r_{t}^{2}.
\end{aligned}
\end{equation}
Here we use \eqref{eqn:pupu}, \eqref{eqn:t+1}, 
and $\|\nabla E(\theta^{(t+1)})\|_2 
\leqslant L r_{t+1} \leqslant L r_{t}$ 
from \Cref{lemma:decom}.
Following \Cref{lemma:decom}, we write 
$\nabla E(\theta^{(t)}) = x^{(t)} + w^{(t)}$ with  $x^{(t)}=\mathcal{P}_{{U^{(t)}}^\perp}\nabla E(\theta^{(t)})$ 
and $w^{(t)}=\mathcal{P}_{U^{(t)}}\nabla E(\theta^{(t)})$ 
satisfying $\|w^{(t)}\|_2\leqslant M r_{t}^2$. 
Moreover, we observe that  
$\mathcal{P}_{{u^{(t)}}^\perp}$ commutes with $F^{(t)}$ and $\mathcal{P}_{{U^{(t)}}^\perp}$. 
Consequently, we derive the following estimates:
\begin{equation}\label{eqn:334}
\begin{aligned}
&\quad \|\mathcal{P}_{{u^{(t)}}^\perp}F^{(t)} 
\nabla E(\theta^{(t)})\|_2
\leqslant 
\|F^{(t)}{\mathcal{P}_{{u^{(t)}}^\perp} x^{(t)}}\|_2
+\|F^{(t)}\|_2
\|w^{(t)}\|_2\\
&\leqslant \|F^{(t)} \mathcal{P}_{{U^{(t)}}^\perp}{\mathcal{P}_{{u^{(t)}}^\perp}}\|_2\cdot
z_t +2\|w^{(t)}\|_2
\leqslant (1-\beta\mu_{2})z_{t}+2Mr_{t}^2.
\end{aligned}
\end{equation}
The last inequality follows from \eqref{eqn:expan} and $\max\left(|\lambda_s^{(t)}|, |\lambda_{s+m}^{(t)}|\right) > \mu_{2}$, which imply
\begin{equation}
\|F^{(t)}\mathcal{P}_{{U^{(t)}}^\perp}{\mathcal{P}_{{u^{(t)}}^\perp}}\|_2
=\|F^{(t)} (I- \mathcal{P}_{{U^{(t)}}}- \mathcal{P}_{{u^{(t)}}})\|_2
\leqslant1-\beta\mu_{2}.
\end{equation}
Combining \eqref{eqn:332} and \eqref{eqn:334}, we obtain 
\begin{equation}\label{eqn:zzz}
z_{t+1}\leqslant (1-\beta\mu_{2})z_{t}+\overline{C}r_{t}^2,
\quad \text{where } \overline{C}=4M+2CL.
\end{equation}

Similarly, for $y_{t+1}$, we obtain 
\begin{equation}\label{eqn:yyy}
\begin{aligned}
y_{t+1}&=
\|\mathcal{P}_{u^{(t+1)}}\nabla E(\theta^{(t+1)})\|_{2}
\\
&\geqslant
\|\mathcal{P}_{u^{(t)}}\nabla E(\theta^{(t+1)})\|_{2}-
\|\mathcal{P}_{u^{(t+1)}}-\mathcal{P}_{u^{(t)}}\|_{2}
\|\nabla E(\theta^{(t+1)})\|_{2}
\\
&\geqslant
\|\mathcal{P}_{u^{(t)}}F^{(t)}
\nabla E(\theta^{(t)})\|_{2}
-\|D^{(t)}(\theta^{(t+1)}-\theta^{(t)})\|_{2}
- 2CLr_{t}^{2}
\\
&\geqslant
\|F^{(t)}\mathcal{P}_{u^{(t)}} \mathcal{P}_{u^{(t)}}\nabla E(\theta^{(t)})\|_{2}
- (2M+2CL)r_{t}^{2}
\\
&\geqslant
\|(1-\beta|\lambda_{\min}^{(t)}|)\mathcal{P}_{u^{(t)}}\nabla E(\theta^{(t)})\|_{2}- \overline{C}r_{t}^{2}
\geqslant
(1-\beta\mu_1)y_{t}-\overline{C}r_{t}^{2}.
\end{aligned}
\end{equation}

From $r_t\leqslant y_t/\mu_0$ in Assumption \ref{ass:power}, we have 
$y_{t+1} \geqslant 
(1-\beta\mu_1)y_t  -\overline{C}r_{t}y_t/\mu_0$.
\Cref{thm:convergence} implies that  $r_t \leqslant C_1 \xi^t$ for some constant $C_1>0$, where $\xi=(1+\mu\beta)^{-1}$. 
Consequently, there exists an integer $T>0$ such that for all $t\geqslant T$, Assumptions \ref{ass:align} and \ref{ass:power} are satisfied, and the condition $\overline{C}r_{t}/\mu_0 < (1-\beta\mu_1)/2$ holds. Therefore, 
\begin{equation}\label{eqn:slt}
y_{t+1} \geqslant
(1-\beta\mu_1)y_t  -\overline{C}r_{t}y_t/\mu_0
> (1-\beta\mu_1) y_t/2 > 0.
\end{equation}
Combining \eqref{eqn:zzz} and \eqref{eqn:yyy}, we obtain
\begin{equation}
\frac{z_{t+1}}{y_{t+1}} \leqslant \frac{(1-\beta\mu_2)z_t + \overline{C}r_t^2}{(1-\beta\mu_1)y_t - \overline{C}r_t^2}
=\rho\frac{z_t}{y_t}
+ 
\frac{(1-\beta\mu_2)z_t + (1-\beta\mu_1)y_t}
{(1-\beta\mu_1)y_t((1-\beta\mu_1)y_t - \overline{C}r_t^2)}\overline{C}r_t^2
,
\end{equation}
where $0 < \beta\mu_1<\beta\mu_2 = {\mu_2}/{L} < 1$, and 
$\rho=(1-\beta\mu_2)/(1-\beta\mu_1)\in(0,1)$. \\
Using $r_t\leqslant y_t/\mu_0$, $z_t \leqslant L r_t \leqslant L y_t/\mu_0$ from \Cref{lemma:decom}, and \eqref{eqn:slt}, we have
\begin{equation}
\frac{z_{t+1}}{y_{t+1}} - \rho\frac{z_t}{y_t} \leqslant
\frac{(1-\beta\mu_2)Ly_t/\mu_0 + (1-\beta\mu_1)y_t}
{(1-\beta\mu_1)^2y_t^2/2}\overline{C}r_t y_t/\mu_0
:=\tilde{C} r_t,
\end{equation}
here $\tilde{C}=2\overline{C}\mu_0^{-2}(1-\beta\mu_1)^{-2}[(1-\beta\mu_2)L+(1-\beta\mu_1)\mu_0]  $.

Finally, for $t\geqslant T$, using $r_t \leqslant C_1\xi^t$, we have 
\begin{equation}\label{eqn:ztyt}
\frac{z_t}{y_t} \leqslant 
\rho^{t-T} \frac{z_T}{y_T} + 
\tilde{C} C_1
\sum_{\tau=T}^{t-1}\rho^{t-1-\tau} \xi^\tau
\leqslant \rho^{t-T} \frac{z_T}{y_T} + \tilde C C_1 t
\left(\max\{\rho,\xi\}\right)^{t-1}\to 0,
\end{equation}
as $t\to\infty$, which completes the proof.
\end{proof}

\begin{remark}
Under the gradient alignment condition $\mathcal{P}_{u^\perp}\nabla E(\theta) \approx 0$, the late-stage decay rate of $\|\nabla E(\theta)\|_2$ can be estimated asymptotically as follows:
\begin{equation*}
\frac{\mathrm{d}}{\mathrm{d}t}G(\theta)
=-(\nabla E(\theta))^{\top}\nabla^{2}E(\theta)\frac{\mathrm{d}\theta}{\mathrm{d}t}
\approx
-|\lambda_{\min}|\cdot \|\nabla E(\theta)\|_2^2,
\end{equation*}
or equivalently,
$\frac{d}{\mathrm{d}t}
\|\nabla E(\theta)\|_2 \approx 
-|\lambda_{\min}| \cdot \|\nabla E(\theta)\|_2$.
Hence, $\|\nabla E(\theta)\|_2$ decays approximately exponentially with rate $-|\lambda_{\min}|$, where $\lambda_{\min}$ is the smallest nonzero eigenvalue of $\nabla^2 E(\theta)$ excluding the $(s+1)$-th to $(s+m)$-th smallest eigenvalues.

In discrete iterations, by applying \Cref{lemma:expansion} to $\hat{\theta}^{(t)}$ and $\theta^{(t)}$, we obtain
\begin{equation}
\nabla E(\theta^{(t)})=
\nabla^{2}E(\theta^{(t)})
(\theta^{(t)}-\hat{\theta}^{(t)})+O(r_t^2).
\end{equation}
Since $\mathcal{P}_{U^{(t)}}(\theta^{(t)}-\hat{\theta}^{(t)})=O(r_t^2)$, we have
\begin{equation}\label{eqn:grad2err}
\theta^{(t)}-\hat{\theta}^{(t)}
= \left(\lambda_{\min}^{(t)}\right)^{-1} \nabla E(\theta^{(t)})+o(r_t), \quad
\text{as }t\to \infty,
\end{equation}
where we used the gradient alignment result \eqref{eqn:ztyt}.
In practice, one can apply $\|\nabla E(\theta)\|_2$ as a terminal condition to monitor the convergence of the numerical algorithm.
The error $\|\theta^{(t)}-\hat{\theta}^{(t)}\|$ can be estimated using \eqref{eqn:grad2err} at the end of iterations.

\end{remark}
\section{Numerical experiments}\label{sec:4}
In the context of energy landscapes with degenerate saddle points, the loss landscape of over-parameterized neural networks exhibits a distinctive level of extreme degeneracy and non-convexity, distinguishing it from simpler models with well-defined symmetries. 
Most critical points of the loss function are degenerate, and the landscape contains multiple local minima alongside diverse saddle points. 
This complexity poses significant challenges for neural-network optimization \cite{NIPS2014}.
Moreover, recent research indicates that gradient-based optimization methods with small initializations can induce a ``saddle-to-saddle'' training regime \cite{Pesme2023Saddle}. 
Consequently, it is essential to understand the impact of saddle points in deep learning. 
Computational analysis of saddle points can offer valuable insights into the optimization dynamics and training behavior of neural networks.

In this section, we present some numerical results for neural networks to demonstrate the performance of HiSD in computing degenerate saddle points.
We adopt an experimental setup of a two-layer neural network with six neurons:
\begin{equation}
f_{\theta}: \mathbb{R} \to \mathbb{R}, \quad
f_{\theta}(x) = \sum_{i=1}^{6} a_i \tanh(w_i x + b_i),
\end{equation}
as the hypothesis space. 
The ground truth $f^*=\sum_{i=1,2} a_i^* \tanh(w_i^* x + b_i^*)$ has two neurons.
In our experiments, the training dataset is $S=\{(x_i,y_i)\}_{i=1}^n$, where $n=15$, with $x_i \sim  \mathcal{N}(0,1)$ i.i.d. and $y_i=f^*(x_i)$. 
The empirical risk, serving as the energy function, is given by
\begin{equation}
E({\theta})=\frac{1}{2n}\sum_{i=1}^{n}\left(f_{\theta}({x}_{i})-y_{i}\right)^{2}.
\end{equation}

It is shown in \cite{zhang2021embedding} that, during neural network training, the parameters may approach certain saddle manifolds, which play an important role in understanding both the training dynamics and the loss landscape. 
Such saddle manifolds can be constructed via neuron splitting. 
To ensure precision in our numerical experiments, we employ this method to construct the saddle manifold $\mathcal{M}$. 
In our numerical experiments, the saddle manifold $\mathcal{M}$ under consideration is a 5‑dimensional linear manifold, where the Hessian matrix has five negative eigenvalues, five zero eigenvalues, and eight positive eigenvalues,
\textit{i.e.}, index $s=5$ and nullity $m=5$ in the notation of this paper.
The initial point $\theta^{(0)}$ is chosen as a random perturbation of a fixed saddle point $\theta^* \in \mathcal{M}$, and this same initial point is used consistently across all experiments.

\subsection{HiSD with different indices}
Our theoretical results show that the $k$-HiSD method exhibits convergence and similar qualitative behaviors for $s\leqslant k \leqslant s+m$.
Therefore, we apply $k$-HiSD for $5 \leqslant k \leqslant 10$ with a fixed step size $\beta = 0.1$ and LOBPCG as the eigensolver.
\Cref{fig:index} shows the 2-norm of gradient, the distance to the saddle manifold, and the extent of gradient alignment during iterations for different $k$.
Notably, since the projection $\hat{\theta}^{(t)} = \mathcal{P}_{\mathcal{M}} \theta^{(t)}$ onto the linear manifold $\mathcal{M}$ can be computed explicitly, we can directly evaluate $\|\theta^{(t)} - \hat{\theta}^{(t)}\|_2$.

\begin{figure}[!htbp]
\centering
\includegraphics[width=1\linewidth]{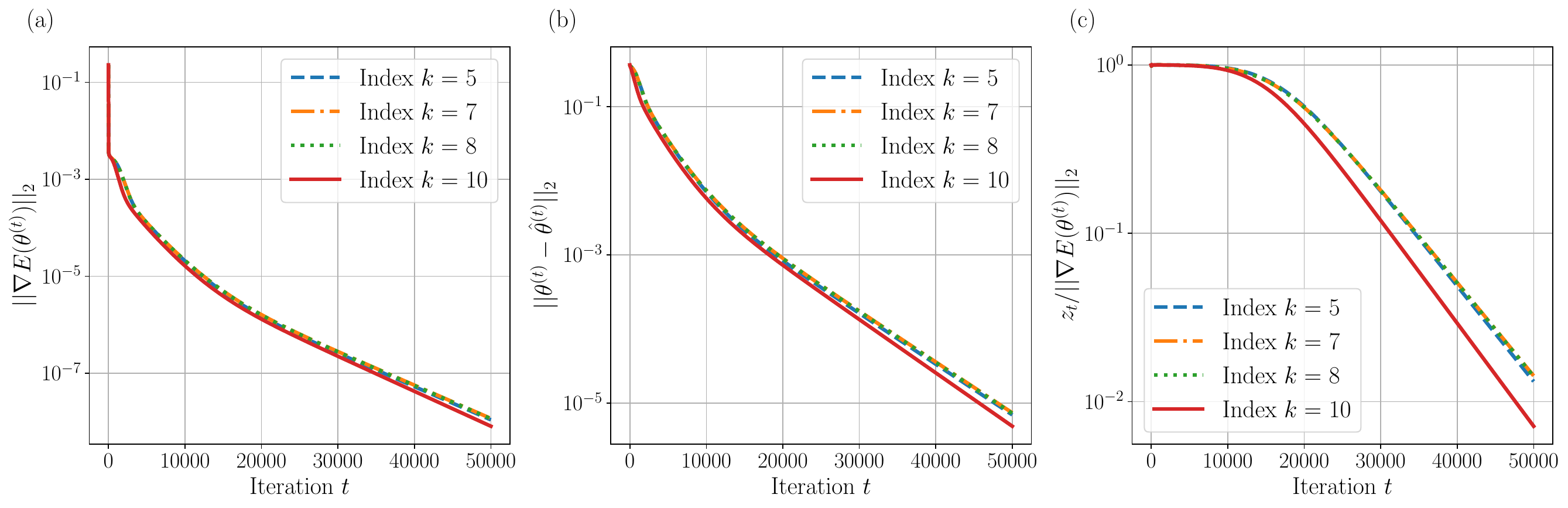}
\caption{Plots of 
(a) gradient norm $\|\nabla E(\theta^{(t)})\|_2$, 
(b) distance to the saddle manifold $\|\theta^{(t)}-\hat{\theta}^{(t)}\|_2$ and 
(c) extent of gradient alignment ${z_t}/{\|\nabla E(\theta^{(t)})\|_2}$ with respect to the iteration number $t$ for HiSD with different indices $k$.}
\label{fig:index}
\end{figure}
\Cref{fig:index} illustrates that both $\|\nabla E(\theta^{(t)})\|_2$ and $\|\theta^{(t)}-\hat{\theta}^{(t)}\|_2$ converge to zero at similar rates for different indices $k$. 
Furthermore, the ratio ${z_t}/{\|\nabla E(\theta^{(t)})\|_2}$, which represents the component of the gradient orthogonal to the eigenvector corresponding to $\lambda_{\min}$, also converges to zero in agreement with the gradient alignment tendency.

\subsection{Different acceleration methods for HiSD}
As demonstrated by the preceding numerical results, saddle point degeneracy poses a significant challenge to search algorithms, often resulting in slow convergence rates and requiring a large number of iterations. 
To address this, momentum‐accelerated variants of HiSD---for example, HiSD with Heavy‐Ball acceleration \cite{luo2025accelerated} and HiSD with Nesterov’s acceleration \cite{liu2024neural}---have been developed, effectively improving convergence speed.
In our numerical experiments, we compare the performance of the explicit Euler scheme, Heavy‐Ball acceleration with momentum parameters $\gamma=0.6$ and $0.9$, and Nesterov’s acceleration. 
All methods utilize a fixed step size of $\beta = 0.1$ and employ LOBPCG as the eigensolver.
For Nesterov’s acceleration, a restart strategy is implemented every 500 iterations. 

\begin{figure}[!htbp]
\centering
\includegraphics[width=1\linewidth]{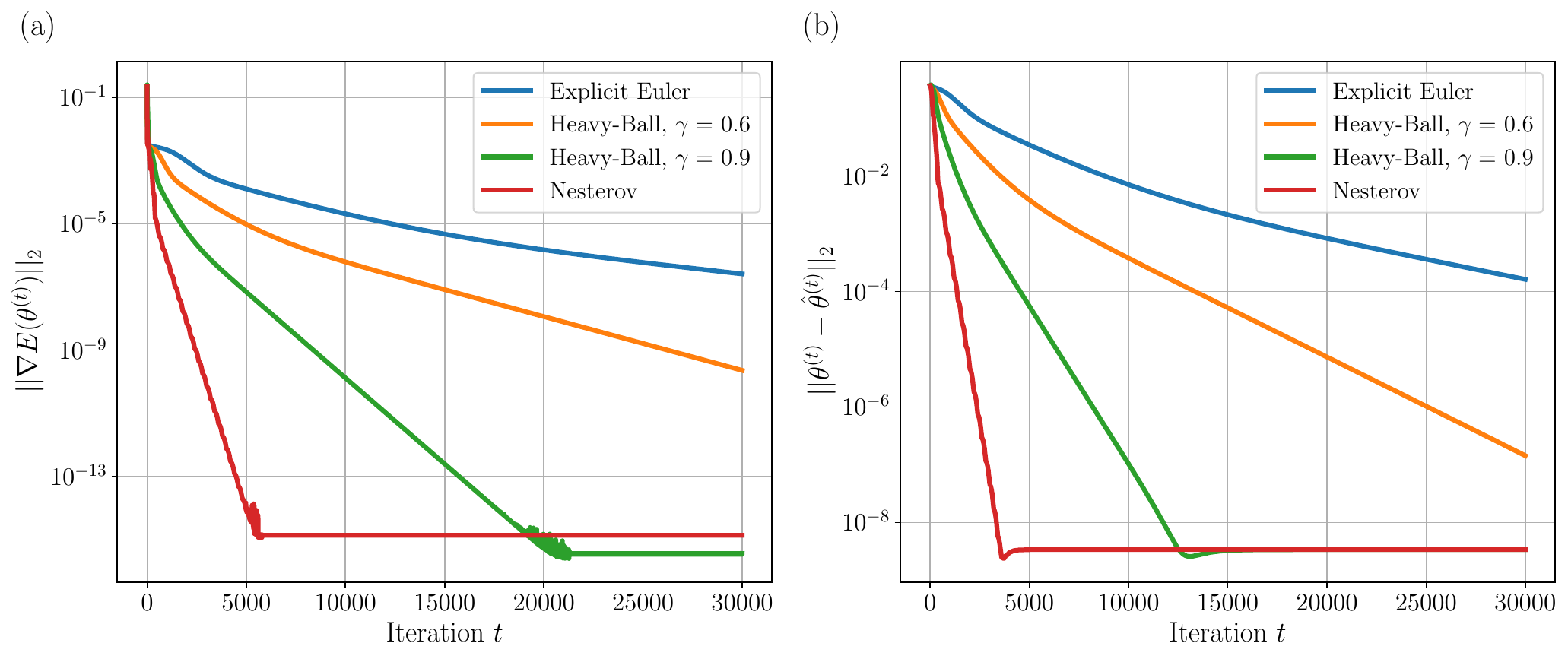}
\caption{Plots of 
(a) gradient norm $\|\nabla E(\theta^{(t)})\|_2$ and 
(b) distance to the saddle manifold $\|\theta^{(t)}-\hat{\theta}^{(t)}\|_2$ with respect to the iteration number $t$ for different acceleration methods.}
\label{fig:acceleration}
\end{figure}

As illustrated in \cref{fig:acceleration}, these accelerated methods significantly improve convergence speed compared to the explicit Euler scheme.
The Nesterov’s acceleration yields superior performance, with $\|\nabla E(\theta^{(t)})\|_2$ decreasing to $10^{-15}$ in approximately 5,000 steps. 
Similarly, the Heavy‐Ball method with $\gamma = 0.9$ reduces the gradient norm to $10^{-15}$ within roughly 20,000 steps, while the one with $\gamma=0.6$ converges at a slower speed. 
The distance to the saddle manifold, $\|\theta^{(t)}-\hat{\theta}^{(t)}\|_2$, exhibits a similar convergence behavior. 

\section{Concluding remarks} \label{sec:5}
In this paper, we have presented a rigorous analysis of HiSD for degenerate saddle points on critical manifolds by utilizing Morse--Bott functions. 
We addressed the challenge posed by zero eigenvalues in the Hessian by proving the local convergence of the continuous HiSD and, most notably, establishing the linear convergence rate of the discrete HiSD algorithm near critical manifolds. 
Additionally, we provided a theoretical justification for the gradient alignment tendency and clarified the flexibility of index selection in practical computations.
Numerical experiments on neural networks validated our analysis, demonstrating that HiSD maintains rapid convergence even in highly-degenerate landscapes. 
These results bridge the gap between the empirical success of HiSD and its mathematical understanding in degenerate cases.

Our findings confirm the feasibility of applying HiSD to search for degenerate saddle points, providing a solid theoretical foundation for future applications in degenerate problems, and further underscoring the broad applicability of HiSD as an efficient saddle-searching method. 
Given that energy functions in practical problems often exhibit degeneracy, the development of more efficient and specialized algorithms remains a priority. 
For instance, recent work has improved HiSD by filtering zero eigenvectors to accelerate convergence \cite{jiang2025nullspace}. 
Moving forward, further advancements and the design of robust algorithms capable of efficiently searching for degenerate saddle points in scientific applications warrant continued research.

\bibliographystyle{siamplain} 
\bibliography{refs}
\appendix

\section{Proofs of lemmas}
\label{app}
\begin{proof}[Proof of \Cref{lemma:dist}]
\cite[Theorem 8.1.5]{golub2013matrix} leads to
\begin{equation}\label{eqn:lemma:dist1}
|\lambda_p-\hat{\lambda}_p| \leqslant 
\|\nabla^2 E(\hat{\theta})-\nabla^2 E(\theta)\|_2 \leqslant
M\|\theta-\hat{\theta}\|_2 \leqslant
M\delta \leqslant \mu/2. 
\end{equation}
From \cite[Theorem 3.6]{stewart1990matrix}, we have
\begin{equation}\label{eqn:lemma:dist2}
\operatorname{dist}(U, \widehat{U}) \leqslant
\dfrac{1}{\mu-\mu/2} \|\nabla^2 E(\hat{\theta})-\nabla^2 E(\theta)\|_2=\dfrac{2}{\mu} \|\nabla^2 E(\hat{\theta})-\nabla^2 E(\theta)\|_2.
\end{equation}
The result can be obtained by combining \eqref{eqn:lemma:dist1} and \eqref{eqn:lemma:dist2}.
\end{proof}

\begin{proof}[Proof of \Cref{lemma:expansion}]
The first part directly comes from the identity
\begin{equation}\label{eqn:identity}
\nabla E(\theta')-\nabla E(\theta)=
\left(\int_{0}^{1}\nabla^{2}E\left(\theta+t(\theta'-\theta)\right) \mathrm{d}t\right)
\left(\theta'-\theta\right),
\end{equation}
If $\theta, \theta' \in \mathcal{U}(\mathcal{M},\delta)$, then by Assumption \ref{ass:1}(i), we obtain
\begin{equation*}
\|J\|_2
\leqslant\int_0^1 
\left\|\nabla^2E(\theta')-\nabla^2 E(\theta+t(\theta'-\theta))\right\|_2 \mathrm{d}t
\leqslant 
\frac{1}{2}M\|\theta'-\theta\|_2. 
\end{equation*}
\end{proof}

\begin{proof}[Proof of \Cref{lemma:decom}]
Using \eqref{eqn:identity} at $\hat\theta$ and $\theta$, we obtain 
\begin{equation}
\|\nabla E(\theta)\|_2
\leqslant
\int_{0}^{1}\left\|\nabla^{2}E(\hat\theta+t(\theta-\hat\theta))\right\|_2\mathrm{d}t\cdot
\|\theta-\hat\theta \|_2\leqslant L\|\theta-\hat\theta\|_2.
\end{equation}
Using \Cref{lemma:expansion} at the base points $\hat\theta$ and $\theta$, we obtain 
\begin{equation}
\nabla E(\theta)=
\nabla^2 E(\theta)\mathcal{P}_{{U}^\perp}(\theta-\hat\theta)+
\nabla^2 E(\theta)\mathcal{P}_{{U}}(\theta-\hat\theta)+ 
J(\theta-\hat\theta),
\end{equation}  
where the first term lies in ${U}^\perp$, the second in $U$, and $\|J\|_2\leqslant \tfrac M2 \|\theta-\hat\theta\|_2$.
Because of $\mathrm{dist}(\widehat{U}, U) \leqslant C \|\theta-\hat{\theta}\|_2$ from \Cref{lemma:dist}, it follows that
\begin{equation*}
\|\mathcal{P}_{{U}}(\theta-\hat\theta)\|_2=
\|\mathcal{P}_{U}(\theta-\hat{\theta})-\mathcal{P}_{\widehat{U}}(\theta-\hat{\theta})\|_2
\leqslant 
\mathrm{dist}(\widehat{U},U) \|\theta-\hat{\theta}\|_2 
\leqslant C\|\theta-\hat{\theta}\|_2^2.
\end{equation*}
Now we derive estimates for 
$\|\mathcal{P}_{{U}^\perp}\nabla E(\theta)\|_2$
and $\|\mathcal{P}_{U}\nabla E(\theta)\|_2$.
Considering that the restriction of $\nabla^2E(\theta)$ to $U^\perp$ has all eigenvalues in $(\-\infty,-\mu]\cup [\mu,+\infty)$, we have
\begin{equation*}
\begin{aligned}
& \quad \|\mathcal{P}_{{U}^\perp}\nabla E(\theta)\|_2
\geqslant
\|\nabla^2 E(\theta)\mathcal{P}_{{U}^\perp}(\theta-\hat\theta)\|_2-\|J(\theta-\hat\theta)\|_2\\
&\geqslant
\mu\left(\|\theta-\hat\theta\|_2-\|\mathcal{P}_{{U}}(\theta-\hat\theta)\|_2\right)-\|J\|_2\|\theta-\hat\theta\|_2
\geqslant\|\theta-\hat\theta\|_2\left(\mu-\tfrac{5}{2}M\|\theta-\hat\theta\|_2\right),   
\end{aligned}
\end{equation*}
From $\|\nabla^2E(\theta)\mathcal{P}_{{U}}\|_2\leqslant\frac{\mu}{4}$  in Remark \ref{re:asm}, we can derive that
\begin{equation}
\begin{aligned}
\|\mathcal{P}_{U}(\nabla E(\theta))\|_2&\leqslant \|\nabla^2E(\theta)\mathcal{P}_{{U}}\|_2\|\mathcal{P}_{{U}}(\theta-\hat\theta)\|_2+\|J\|_2\|\theta-\hat\theta\|_2\\
&\leqslant\left(\frac{\mu}{4}C+\frac M2\right)\|\theta-\hat\theta\|_2^2=M\|\theta-\hat\theta\|_2^2,
\end{aligned}
\end{equation}  
\end{proof}

\end{document}